\documentclass[12pt]{amsproc}

\usepackage{fullpage}
\usepackage[utf8]{inputenc}
\usepackage[T1]{fontenc}
\usepackage{graphicx,subfigure}
\usepackage{amsmath,amsfonts,amssymb,mathtools}
\usepackage{stmaryrd}
\usepackage{mathrsfs,dsfont}
\usepackage{amsthm}
\usepackage{mathabx}
\usepackage{tabularx}
\usepackage{float}

\usepackage{hyperref}

\usepackage{enumerate}

\usepackage{color}

\newcommand{\E}{\mathbb{E}}
\newcommand{\R}{\mathbb{R}}

\newcommand{\T}{\mathbb{T}}

\newtheorem{theo}{Theorem}[section]

\newtheorem{ass}{Assumption}

\usepackage{algorithm}
\usepackage{algorithmic}

\begin{document}

\title{The averaging principle for stochastic differential equations driven by a Wiener process revisited}

\author{Charles-Edouard Br\'ehier}
\address{Univ Lyon, Université Claude Bernard Lyon 1, CNRS UMR 5208, Institut Camille Jordan, 43 blvd. du 11 novembre 1918, F-69622 Villeurbanne cedex, France}
\email{brehier@math.univ-lyon1.fr}

\date{}


\begin{abstract}
We consider a one-dimensional stochastic differential equation driven by a Wiener process, where the diffusion coefficient depends on an ergodic fast process. The averaging principle is satisfied: it is well-known that the slow component converges in distribution to the solution of an averaged equation, with generator determined by averaging the square of the diffusion coefficient.

We propose a version of the averaging principle, where the solution is interpreted as the sum of two terms: one depending on the average of the diffusion coefficient, the other giving fluctuations around that average. Both the average and fluctuation terms contribute to the limit, which illustrates why it is required to average the square of the diffusion coefficient to find the limit behavior.
\end{abstract}

\maketitle

\section{Introduction}\label{sec:intro}

Multiscale and stochastic systems are ubiquitous in all fields of science and engineering. Averaging and homogenization techniques~\cite{PavliotisStuart} are popular methods to derive lower dimensional problems, which are easier to understand and simulate. In this article, we focus on the averaging principle for the following class of stochastic differential equations (SDEs)
\begin{equation}\label{eq:SDEintro}
dX^\epsilon(t)=\sigma\bigl(X^\epsilon(t),m(t/\epsilon)\bigr)d\beta(t),
\end{equation}
where $\epsilon\ll 1$ is the time scale separation parameter, $\beta$ is a standard real-valued Brownian motion, and the diffusion coefficient $\sigma$ is a smooth function. See Section~\ref{sec:ass} for precise assumptions. The fast component of the system is given by an ergodic Markov process $\bigl(m(t)\bigr)_{t\ge 0}$, evolving at the time scale $t/\epsilon$. The averaging principle states that one can eliminate the fast process when $\epsilon\to 0$, precisely the slow component $X^\epsilon$ converges (in distribution) to the solution $\overline{X}$ of an autonomous evolution equation called the averaged equation. In the case of the system~\eqref{eq:SDEintro}, the averaged equation is a SDE of the type
\begin{equation}\label{eq:averagedintro}
d\overline{X}(t)=\Sigma(\overline{X}(t))d\beta(t),
\end{equation}
where
\[
\Sigma^2(\cdot)=\overline{\sigma^2}(\cdot)=\int \sigma(\cdot,m)^2d\mu(m), 
\]
and $\mu$ denotes the invariant probability distribution of the fast ergodic process $\bigl(m(t)\bigr)_{t\ge 0}$.

In this article, we revisit this problem, and propose an original point of view which explains why the limit equation is not given by simply averaging the diffusion coefficient $\sigma$, which would give
\[
dX(t)=\overline{\sigma}(X(t))d\beta(t).
\]
Note that one has $\Sigma^2=\overline{\sigma^2}\ge \overline{\sigma}^2$, thus the averaging principle may be interpreted as exhibiting enhanced diffusion. The approach used in this article can be explained as follows: we introduce a decomposition $X^\epsilon=Y^\epsilon+Z^\epsilon$ of the slow component, where
\begin{equation}\label{eq:systemintro}
\left\lbrace
\begin{aligned}
dY^\epsilon(t)&=\overline{\sigma}(Y^\epsilon(t)+Z^\epsilon(t))d\beta(t)\\
dZ^\epsilon(t)&=\bigl(\sigma(Y^\epsilon(t)+Z^\epsilon(t),m^\epsilon(t))-\overline{\sigma}(Y^\epsilon(t)+Z^\epsilon(t))\bigr)d\beta(t).
\end{aligned}
\right.
\end{equation}
Observe that $Y^\epsilon$ is defined in terms of $\overline{\sigma}$, and thus may be interpreted as an average term, whereas $Z^\epsilon$ may be interpreted as a fluctuation term. The reason behind the expression of the averaged equation~\eqref{eq:averagedintro} in terms of $\overline{\sigma^2}$ is the fact that $Z^\epsilon$ converges to a non-trivial limit when $\epsilon\to 0$. Precisely, the main result of this article, Theorem~\ref{th:main}, states that $(Y^\epsilon(T),Z^\epsilon(T))$ converges in distribution, when $\epsilon\to 0$, to $(Y(T),Z(T))$, for all $T\ge 0$, given by
\begin{equation}\label{eq:system_avintro}
\left\lbrace
\begin{aligned}
dY(t)&=\overline{\sigma}(Y(t)+Z(t))d\beta_t^1\\
dZ(t)&=\langle \sigma\rangle (Y(t)+Z(t))d\beta_t^2,
\end{aligned}
\right.
\end{equation}
where $\bigl(\beta_t^1\bigr)_{t\ge 0}$ and $\bigl(\beta_t^2\bigr)_{t\ge 0}$ are two independent standard real-valued Wiener processes, and $\langle\sigma\rangle^2=\overline{(\sigma-\overline{\sigma})^2}$. It is then straightforward to retrieve the standard version of the averaging principle: $X^\epsilon(T)=Y^\epsilon(T)+Z^\epsilon(T)\to Y(T)+Z(T)$, and one checks that $Y(T)+Z(T)$ is equal to $\overline{X}(T)$ in distribution. That identity is due to the following observation: one has $\overline{\sigma}^2+\langle \sigma\rangle^2=\overline{\sigma^2}$. The decomposition into average and fluctuation terms then clearly explains the diffusion enhancement in the averaged equation~\eqref{eq:averagedintro}.

The main result of this article has an elementary formulation. Even if the averaging principle has been extensively studied by many authors, to the best of our knowledge, it seems that the point of view proposed in this article is original and that Theorem~\ref{th:main} is a new result in the mathematical literature. The analysis is performed for a simple one-dimensional SDE, it may be generalized to more complicated problems.

Let us review the literature concerning the averaging principle for SDEs. The list of references is not exhaustive. We refer to the seminal article~\cite{Hasminskii} by Hasminkskii and to the standard monograph~\cite{FreidlinWentzell} (in particular Chapter~7). See also~\cite{PavliotisStuart} (in particular Chapter~17) for a recent overview of the averaging and homogenization techniques for SDEs. Let us also mention~\cite{Veretennikov}, and the recent works~\cite{RocknerSunXie:19,RocknerXie:20}. In the last decade, the averaging principle has been extensively studied for systems of stochastic partial differential equations, see for instance~\cite{Cerrai,CerraiFreidlin}, contributions of the author~\cite{B1,B2} and references therein. Recently Hairer and Li~\cite{HairerLi} have extended the averaging principle for SDE systems of the type~\eqref{eq:SDEintro} where the standard Brownian motion $\beta$ is replaced by a fractional Brownian motion $\beta^H$ with Hurst index $H>1/2$: in Section~\ref{sec:disc} below we explain how the point of view developped in the present article is related to that generalization. Finally, numerical methods for systems of the type~\eqref{eq:SDEintro} which are efficient when $\epsilon\ll 1$ have been studied: see for instance the heterogeneous multiscale method proposed in~\cite{ELiuVandenEijnden:05} and the asymptotic preserving schemes proposed in~\cite{BRR}.

The proof of the main result Theorem~\ref{th:main} employs two standard tools when studying the behavior of multiscale stochastic systems: solutions of Kolmogorov and Poisson equations. We refer for instance to~\cite{KhasminskiiYin:05} and to the series of articles~\cite{PardouxVeretennikov1,PardouxVeretennikov2,
PardouxVeretennikov3} for similar computations. See also~\cite{B1,B2} where weak error estimates in the averaging principle for SPDEs are proved using such techniques, and~\cite{RocknerSunXie:19,RocknerXie:20}. An original feature of the proof of Theorem~\ref{th:main} below is to consider the solutions of two Poisson equations (one related to the average behavior, one related to the fluctuations), whereas the standard approach to the averaging principle only requires a single Poisson equation. This may be surprising since the system~\eqref{eq:SDEintro} only depends on two time scales $t$ and $t/\epsilon$. The use of two Poisson equations is standard in homogenization or diffusion approximation problems, where three time scales $t$, $t/\epsilon$ and $t/\epsilon^2$ appear, see for instance~\cite{CotiZelatiPavliotis,HairerPavliotis} and \cite{PardouxVeretennikov1,PardouxVeretennikov2,
PardouxVeretennikov3}. The list of references is not exhaustive.

This article is organized as follows. Section~\ref{sec:setting} is devoted to state the assumptions (Section~\ref{sec:ass}) and the main result (Section~\ref{sec:main}) of this article. The proof of Theorem~\ref{th:main} is provided in Section~\ref{sec:proof}. Concluding remarks and perspectives for future works are given in Section~\ref{sec:disc}.

\section{Setting and main result}\label{sec:setting}

Let $\bigl(\beta(t)\bigr)_{t\ge 0}$ be a standard real-valued Wiener process. Let $\epsilon\in(0,1)$ denote the time-scale separation parameter. We consider the following SDE on the one-dimensional torus $\T$
\begin{equation}\label{eq:SDE}
dX^\epsilon(t)=\sigma\bigl(X^\epsilon(t),m^\epsilon(t)\bigr)d\beta(t),
\end{equation}
with initial condition $X_0^\epsilon=x_0\in\R$ (assumed to be deterministic and independent of $\epsilon$ for simplicity). Assumptions for the diffusion coefficient $\sigma$ and the fast process $m^\epsilon$ are given in Section~\ref{sec:ass} below.

Working in the one-dimensional torus $\T$ simplifies the presentation, however one may replace $\T$ by $\R$ with minor modifications in the setting. Generalization to higher dimensional problems is mentioned in Section~\ref{sec:disc}.

\subsection{Assumptions}\label{sec:ass}

The diffusion coefficient $\sigma$ is assumed to satisfy the following conditions.
\begin{ass}\label{ass:sigma}
The mapping $\sigma:\T\times\R\to\R$ is of class $\mathcal{C}^4$, with bounded derivatives with respect to the second variable $m$. In addition, assume that for all $x\in\T$, the mapping $\sigma(x,\cdot)$ is not constant.
\end{ass}
In particular, note that $\sigma$ is Lipschitz continuous, this ensures the global well-posedness of~\eqref{eq:SDE} for all $\epsilon>0$.

The fast process $m^\epsilon$ is assumed to satisfy the following conditions:
\begin{ass}\label{ass:m}
For all $\epsilon\in(0,1)$ and all $t\ge 0$, one has $m^\epsilon(t)=m(t/\epsilon)$, where $\bigl(m(t)\bigr)_{t\ge 0}$ is a real-valued ergodic Markov process which is independent of $\beta$. We assume that the initial condition $m(0)=m_0$ is a given deterministic real number. Assume that $\underset{t\ge 0}\sup~\E[|m(t)|^2]<\infty$.

Let $\mu$ denote the unique invariant probability distribution of the process $\bigl(m(t)\bigr)_{t\ge 0}$, and let $\mathcal{L}$ denote its infinitesimal generator.

Define, for all $x\in\T$,
\begin{equation}\label{eq:bars}
\left\lbrace
\begin{aligned}
\overline{\sigma}(x)&=\int \sigma(x,m)d\mu(m)\\
\langle \sigma\rangle(x)&=\sqrt{\int \bigl(\sigma(x,m)-\overline{\sigma}(x)\bigr)^2 d\mu(m)}.
\end{aligned}
\right.
\end{equation}

We assume that for all $x\in\R$, the Poisson equations
\begin{equation}\label{eq:Poisson}
\left\lbrace
\begin{aligned}
-\mathcal{L}\psi_1(x,\cdot)&=\sigma(x,\cdot)-\overline{\sigma}(x)\\
-\mathcal{L}\psi_2(x,\cdot)&=\bigl(\sigma(x,\cdot)-\overline{\sigma}(x)\bigr)^2-\langle \sigma\rangle(x)^2
\end{aligned}
\right.
\end{equation}
admit solutions $\psi_1$, $\psi_2$ -- without loss of generality one assumes that for all $x\in\T$ one has $\int\psi_1(x,m)d\mu(m)=\int\psi_2(x,m)d\mu(m)=0$ -- and that the solutions $\psi_1,\psi_2$ are of class $\mathcal{C}^4$ on $\T\times\R$. In addition, the derivatives are assumed to grow at most quadratically with respect to $m$.
\end{ass}

Note that the mappings $\overline{\sigma}$ and $\langle\sigma\rangle^2$ inherit the regularity properties from the mapping $\sigma$ with respect to the $x$-variable: in particular they are of class $\mathcal{C}^4$ on the torus $\T$. Recall that for all $x\in\T$ the mapping $\sigma(x,\cdot)$ is not constant (owing to Assumption~\ref{ass:sigma}), thus one has $\langle \sigma\rangle^2(x)>0$ for all $x\in\T$. As a consequence, $\langle\sigma\rangle$ then inherits the regularity properties from $\langle\sigma\rangle^2$, in particular it is of class $\mathcal{C}^4$.

Note that the solvability of the Poisson equations~\eqref{eq:Poisson} is possible since the right-hand sides satisfy the required centering conditions by definitions~\eqref{eq:bars} of $\overline{\sigma}$ and $\langle\sigma\rangle^2$. Observe also that for all $x\in\T$ one has
\begin{equation}\label{eq:bar2}
\overline{\sigma}(x)^2+\langle\sigma\rangle(x)^2=\overline{\sigma^2}(x)=\int \sigma(x,m)^2 d\mu(m).
\end{equation}

Let us provide a standard example for the fast process: $\bigl(m(t)\bigr)_{t\ge 0}$ can be the solution of the SDE
\[
dm(t)=-V'(m(t))dt+\sqrt{2}dW(t),
\]
with appropriate assumptions on the potential $V:\R\to\R$ -- for instance $V(x)=x^2/2$, which gives an Ornstein-Uhlenbeck process. In that example the fast process $\bigl(m^\epsilon(t)\bigr)_{t\ge 0}$ solves the SDE
\[
dm^\epsilon(t)=-\frac{V'(m^\epsilon(t))}{\epsilon}dt+\frac{\sqrt{2}}{\sqrt{\epsilon}}dW(t),
\]
and the invariant distribution $\mu$ is given by
\[
d\mu(m)=Z^{-1}\exp(-V(m))dm
\]
with the normalization constant $Z=\int \exp(-V(m))dm$. In that example, it is straightforward to check that the conditions in Assumption~\ref{ass:m} are satisfied (with appropriate regularity and growth assumptions on $V'$).

\subsection{Main result}\label{sec:main}

The objective of this article is to propose a version of the averaging principle with an original point of view. First, recall that the standard version states that when $\epsilon\to 0$, the solution $X^\epsilon$ of~\eqref{eq:SDE} converges in distribution to the solution $\overline{X}$ of the averaged equation
\begin{equation}\label{eq:averaged}
d\overline{X}(t)=\Sigma(\overline{X}(t))d\beta(t)
\end{equation}
with initial condition $\overline{X}(0)=x_0$, where
\[
\Sigma(x)=\sqrt{\overline{\sigma^2(x)}}.
\]
Note that $\Sigma(x)>\overline{\sigma}(x)$, owing to the identity~\eqref{eq:bar2} and Assumption~\ref{ass:sigma}. We refer for instance to~\cite[Chapter~17]{PavliotisStuart} (and the other references mentioned in Section~\ref{sec:intro}).

The version of the averaging principle studied in this article requires to introduce two auxiliary processes $Y^\epsilon$ and $Z^\epsilon$ as follows: we consider the system
\begin{equation}\label{eq:system}
\left\lbrace
\begin{aligned}
dY^\epsilon(t)&=\overline{\sigma}(Y^\epsilon(t)+Z^\epsilon(t))d\beta(t)\\
dZ^\epsilon(t)&=\bigl(\sigma(Y^\epsilon(t)+Z^\epsilon(t),m^\epsilon(t))-\overline{\sigma}(Y^\epsilon(t)+Z^\epsilon(t))\bigr)d\beta(t)
\end{aligned}
\right.
\end{equation}
with initial conditions $Y^\epsilon(0)=0$ and $Z^\epsilon(0)=x_0$. Observe that by construction, one has the identity
\[
X^\epsilon(t)=Y^\epsilon(t)+Z^\epsilon(t)
\]
for all $t\ge 0$.

The main result of this article is the convergence in distribution of $(Y^\epsilon(T),Z^\epsilon(T))$ to $({Y}(T),{Z}(T))$, where the process $\bigl({Y}(t),{Z}(t)\bigr)_{t\ge 0}$ is defined as follows:
\begin{equation}\label{eq:system_av}
\left\lbrace
\begin{aligned}
dY(t)&=\overline{\sigma}(Y(t)+Z(t))d\beta_t^1\\
dZ(t)&=\langle \sigma\rangle (Y(t)+Z(t))d\beta_t^2,
\end{aligned}
\right.
\end{equation}
where $\bigl(\beta_t^1\bigr)_{t\ge 0}$ and $\bigl(\beta_t^2\bigr)_{t\ge 0}$ are two independent standard real-valued Wiener processes.

We are now in position to state the refined version of the averaging principle.
\begin{theo}\label{th:main}
For all $T\in(0,\infty)$, one has the convergence in distribution
\[
(Y^\epsilon(T),Z^\epsilon(T))\underset{\epsilon\to 0}\to (Y(T),Z(T)).
\]
\end{theo}
Note that the standard version of the averaging principle is a straightforward corollary of Theorem~\ref{th:main}. On the one hand, one has the almost sure equality $X^\epsilon(T)=Y^\epsilon(T)+Z^\epsilon(T)$. On the other hand, set $X(t)=Y(t)+Z(t)$, then one has
\[
dX(t)=\overline{\sigma}(X(t))d\beta_t^1+\langle\sigma\rangle(X(t))d\beta_t^2.
\]
The associated infinitesimal generator is given by
\[
\frac{1}{2}\bigl(\overline{\sigma}(x)^2+\langle\sigma\rangle(x)^2\bigr)\partial_{xx}^2=\frac12 \overline{\sigma^2}(x)\partial_{xx}^2=\frac12\Sigma(x)^2\partial_{xx}^2,
\]
owing to the identity~\eqref{eq:bar2}. As a consequence $X$ and $\overline{X}$ are Markov processes with the same infinitesimal generator, and $X(0)=\overline{X}(0)=x_0$: we thus obtain the equality $X(T)=\overline{X}(T)$ in distribution. Finally, Theorem~\ref{th:main} implies
\[
X^\epsilon(T)=Y^\epsilon(T)+Z^\epsilon(T)\underset{\epsilon\to 0}\to Y(T)+Z(T)=\overline{X}(T)
\]
where the convergence and the equality are understood to hold in distribution.

The refined version is an explanation of the well-known fact that $\Sigma(x)>\overline{\sigma}(x)$ -- which is often justified by the observation that one needs to average the infinitesimal generator of the process instead of its coefficients. It also illustrates why the convergence only holds in distribution. To the best of our knowledge, Theorem~\ref{th:main} is a new result.

Let us present a simplified case to illustrate Theorem~\ref{th:main}: assume that $\sigma(x,m)=\sigma(m)$ only depends on $m$. In that case, $\overline{\sigma}$ and $\langle\sigma\rangle$ are constants, the system~\eqref{eq:system} is rewritten as
\[
\left\lbrace
\begin{aligned}
dY^\epsilon(t)&=\overline{\sigma}d\beta(t)\\
dZ^\epsilon(t)&=\bigl(\sigma(m^\epsilon(t))-\overline{\sigma}\bigr)d\beta(t).
\end{aligned}
\right.
\]
In particular, the distribution of $Y^\epsilon(T)$ is $\mathcal{N}(0,\overline{\sigma}^2)$ does not depend on $\epsilon$. Owing to Theorem~\ref{th:main}, $Z^\epsilon(T)$ converges in distribution to $Z(T)\sim \mathcal{N}(0,\langle \sigma\rangle^2 T)$. In fact, more precisely $(Y^\epsilon(T),Z^\epsilon(T))$ converges in distribution to the non-degenerate Gaussian distribution $\mathcal{N}(0,Q)$ with diagonal covariance matrix $Q$, such that $Q_{11}=\overline{\sigma}^2T$, $Q_{22}=\langle\sigma\rangle^2T$. Finally, $Y^\epsilon(T)+Z^\epsilon(T)$ converges in distribution to $\mathcal{N}(0,\overline{\sigma^2}T)$, since $\overline{\sigma}^2+\langle \sigma\rangle^2=\overline{\sigma^2}$. This confirms how Theorem~\ref{th:main} is a refinement of the standard averaging principle in the simplified case

\section{Proof of Theorem~\ref{th:main}}\label{sec:proof}

The objective of this section is to give the proof of Theorem~\ref{th:main}. Before proceeding, let us first introduce some of the main arguments of the proof.

Assume that $\varphi:\mathbb{T}^2\to\R$ is a mapping of class $\mathcal{C}^4$. We prove below that the weak error satisfies
\begin{equation}\label{eq:weakerror}
\big|\E[\varphi(Y^\epsilon(T),Z^\epsilon(T))]-\E[\varphi(Y(T),Z(T))]\big|\le C(T,\varphi,x_0)\epsilon
\end{equation}
for some $C(T,\varphi,x_0)\in(0,\infty)$. By a standard approximation argument, the weak error estimate~\eqref{eq:weakerror} implies that one has
\[
\E[\varphi(Y^\epsilon(T),Z^\epsilon(T))]\underset{\epsilon\to 0}\to\E[\varphi(Y(T),Z(T))]
\]
for all continuous mappings $\varphi:\T^2\to\R$, which means the convergence in distribution stated in Theorem~\ref{th:main}. It thus suffices to establish the weak error estimate~\eqref{eq:weakerror} to prove Theorem~\ref{th:main}.

To prove the weak error estimate~\eqref{eq:weakerror}, it is convenient to introduce two auxiliary mappings $u$ and $\Phi$ from $[0,T]\times \T^2\times\R$ to $\R$. First, $u$ is the solution of the Kolmogorov equation associated with the SDE system~\eqref{eq:system_av} for $\bigl(Y(t),Z(t)\bigr)_{t\ge 0}$:
\begin{equation}\label{eq:Kolmo}
\partial_tu(t,y,z)=\frac12\overline{\sigma}(y+z)^2\partial_{yy}^2u(t,y,z)+\frac12\langle\sigma\rangle(y+z)^2\partial_{zz}^2u(t,y,z),
\end{equation}
with initial condition $u(0,y,z)=\varphi(y,z)$ for all $(y,z)\in\T^2$. Using Assumption~\ref{ass:sigma} and~\eqref{eq:bars}, one checks that $u$ is of class $\mathcal{C}^4$ with respect to $(y,z)$ and of class $\mathcal{C}^1$ with respect to $t$.

Second, for all $t\ge 0$, $(y,z)\in\T^2$ and $m\in\R$, set
\begin{equation}\label{eq:Phi}
\Phi(t,y,z,m)=\overline{\sigma}(y+z)\partial_{yz}^2u(T-t,y,z)\psi_1(y,z,m)+\frac12\partial_{zz}^2u(T-t,y,z)\psi_2(y,z,m).
\end{equation}
One checks that $\Phi$ is of class $\mathcal{C}^1$ with respect to $t$, and of class $\mathcal{C}^2$ with respect to $(y,z,m)$, with at most quadratic growth with respect to $m$.

\begin{proof}[Proof of Theorem~\ref{th:main}]
Expressing the weak error in terms of the solution $u$ of the Kolmogorov equation~\eqref{eq:Kolmo}, and applying It\^o's formula, one obtains
\begin{align*}
\E[\varphi&(Y_T^\epsilon,Z_T^\epsilon)]-\E[\varphi(Y_T,Z_T)]\\
&=\E[u(0,Y_T^\epsilon,Z_T^\epsilon)]-\E[u(T,Y_0^\epsilon,Z_0^\epsilon)]\\
&=\int_0^T\E\bigl[-\partial_tu(T-t,Y_t^\epsilon,Z_t^\epsilon)\bigr]dt\\
&+\frac{1}{2}\int_0^T\E\bigl[\overline{\sigma}(Y_t^\epsilon+Z_t^\epsilon)^2\partial_{yy}u(T-t,Y_t^\epsilon,Z_t^\epsilon)\bigr]dt\\
&+\int_0^T\E\bigl[\overline{\sigma}(Y_t^\epsilon+Z_t^\epsilon)\bigl(\sigma(Y_t^\epsilon+Z_t^\epsilon,m_t^\epsilon)-\overline{\sigma}(Y_t^\epsilon+Z_t^\epsilon)\bigr)\partial_{yz}^2u(T-t,Y_t^\epsilon,Z_t^\epsilon) \bigr]dt\\
&+\frac12\int_0^T\E\bigl[\bigl(\sigma(Y_t^\epsilon+Z_t^\epsilon,m_t^\epsilon)-\overline{\sigma}(Y_t^\epsilon+Z_t^\epsilon)\bigr)^2\partial_{zz}^2u(T-t,Y_t^\epsilon,Z_t^\epsilon)\bigr]dt\\
&=\int_0^T\E\bigl[\overline{\sigma}(Y_t^\epsilon+Z_t^\epsilon)\Bigl(\sigma(Y_t^\epsilon+Z_t^\epsilon,m_t^\epsilon)-\overline{\sigma}(Y_t^\epsilon+Z_t^\epsilon)\Bigr)\partial_{yz}^2u(T-t,Y_t^\epsilon,Z_t^\epsilon) \bigr]dt\\
&+\frac12\int_0^T\E\bigl[\Bigl(\bigl(\sigma(Y_t^\epsilon+Z_t^\epsilon,m_t^\epsilon)-\overline{\sigma}(Y_t^\epsilon+Z_t^\epsilon)\bigr)^2-\langle \sigma\rangle(Y_t^\epsilon+Z_t^\epsilon)^2\Bigr)\partial_{zz}^2u(T-t,Y_t^\epsilon,Z_t^\epsilon)\bigr]dt,
\end{align*}
where the last line comes from replacing $\partial_tu$ using the Kolmogorov equation~\eqref{eq:Kolmo}.

Observe that the two terms in the right-hand side above have a nice form, since the factors in parenthesis are centered with respect to the invariant distribution $\mu$ in the $m$ variable, and the other factors do not depend on $m$. Recall that the auxiliary functions $\psi_1$ and $\psi_2$ are defined as solutions of the Poisson equations~\eqref{eq:Poisson}. As a consequence, by the definition~\eqref{eq:Phi} of the auxiliary function $\Phi$, the weak error satisfies the identity
\begin{equation}\label{eq:expressweakerror}
\E[\varphi(Y_T^\epsilon,Z_T^\epsilon)]-\E[\varphi(Y_T,Z_T)]=-\int_{0}^{T}\E[\mathcal{L}\Phi(t,Y_t^\epsilon,Z_t^\epsilon,m_t^\epsilon)]dt.
\end{equation}
Applying It\^o's formula, one has
\begin{align*}
\E\bigl[\Phi(T,Y_T^\epsilon,Z_T^\epsilon,m_T^\epsilon)\bigr]&=\E[\Phi(0,Y_0^\epsilon,Z_0^\epsilon,m_0^\epsilon)]\\
&\quad+\int_{0}^{T}\E\bigl[\mathcal{A}\Phi(t,Y_t^\epsilon,Z_t^\epsilon,m_t^\epsilon)\bigr]dt+\frac{1}{\epsilon}\int_{0}^{T}\E[\mathcal{L}\Phi(t,Y_t^\epsilon,Z_t^\epsilon,m_t^\epsilon)]dt,
\end{align*}
where the auxiliary differential operator $\mathcal{A}$ is given by
\[
\mathcal{A}=\partial_t+\frac12\overline{\sigma}(y+z)^2\partial_{yy}^2+\overline{\sigma}(y+z)\bigl(\sigma(y+Z,m)-\overline{\sigma}(y+z)\bigr)\partial_{yz}^2+\frac12\bigl(\sigma(y+Z,m)-\overline{\sigma}(y+z)\bigr)^2\partial_{zz}^2.
\]
Finally, the weak error estimate satifies
\begin{align*}
\E[\varphi(Y_T^\epsilon,Z_T^\epsilon)]-\E[\varphi(Y_T,Z_T)]&=\epsilon\bigl(\E[\Phi(0,Y_0^\epsilon,Z_0^\epsilon,m_0^\epsilon)]-\E[\Phi(T,Y_T^\epsilon,Z_T^\epsilon,m_T^\epsilon)]\bigr)\\
&+\epsilon\int_{0}^{T}\E\bigl[\mathcal{A}\Phi(t,Y_t^\epsilon,Z_t^\epsilon,m_t^\epsilon)\bigr]dt\\
&={\rm O}(\epsilon)
\end{align*}
using the regularity properties of $\Phi$ and the moment estimate
\[
\underset{\epsilon\in(0,1)}\sup~\underset{t\ge 0}\sup~\E[|m^\epsilon(t)|^2]=\underset{t\ge 0}\sup~\E[|m(t)|^2]<\infty
\]
owing to Assumption~\ref{ass:m}.

This concludes the proof of the weak error estimate~\eqref{eq:weakerror} and of Theorem~\ref{th:main}.
\end{proof}

Observe that the proof of Theorem~\ref{th:main} requires to exploit the solutions $\psi_1$ and $\psi_2$ of two auxiliary Poisson equation. On the one hand, the proof of the standard averaging principle exploits the solution $\psi$ of a single Poisson equation, namely
\[
-\mathcal{L}\psi(x,\cdot)=\sigma^2(x,\cdot)-\overline{\sigma^2}(x).
\]
On the other hand, using the solutions of two Poisson equations is standard in homogenization theory, where the infinitesimal generator has an expansion of the form $\mathcal{L}^\epsilon=\mathcal{L}_0+\epsilon^{-1}\mathcal{L}_1+\epsilon^{-2}\mathcal{L}$ -- whereas it is of the form $\mathcal{L}^\epsilon=\mathcal{L}_0+\epsilon^{-1}\mathcal{L}$ in the averaging regime we consider. The two Poisson equation appears to deal with different scales $\epsilon^{0}$ and $\epsilon^{-1}$ in that problem.



\section{Discussion}\label{sec:disc}

In this article, we have revisited the averaging principle for the class of stochastic differential equations given by~\eqref{eq:SDEintro}. Contrary to the standard approach, we propose to decompose $X^\epsilon=Y^\epsilon+Z^\epsilon$ (see~\eqref{eq:system}), where $Y^\epsilon$ is defined in terms of the average $\overline{\sigma}$ (with respect to the fast variable) of the diffusion coefficient, and $Z^\epsilon$ represents fluctuations around the average. Our main result, Theorem~\ref{th:main}, states that $(Y^\epsilon,Z^\epsilon)$ converges in distribution to a non-trivial limit $(Y,Z)$. The key observation is that $Z$ is not equal to $0$, this explains why the limit $\overline{X}$ for $X^\epsilon$ is defined in terms of the average $\overline{\sigma^2}$ of the square of the diffusion coefficient. Note that $\overline{\sigma^2}\ge \overline{\sigma}^2$ by the Cauchy-Schwarz inequality (see~\eqref{eq:bar2}), {\it i.e.} diffusion is enhanced in the averaging procedure, and the behavior of the fluctuation term $Z^\epsilon$ quantifies the increase in the diffusion.

The approach to prove Theorem~\ref{th:main} is based on a classical strategy when studying multiscale stochastic systems: weak error estimates are proved using solutions of the Kolmogorov equation associated with the limit, and of Poisson equations associated with the behavior of the fast component. The solvability of the Poisson equations requires centering conditions to be satisfied, which identify limit. The proof of Theorem~\ref{th:main} is original since we employ the solutions of two Poisson equations, instead of only one in the standard proof of the averaging principle.

Our study is limited to one-dimensional SDEs. It is expected that generalizing the result to higher-dimensional SDEs and SPDEs is possible. This may be studied in future works. Note also that it would be straightforward to include drift terms in the SDE~\eqref{eq:SDEintro}: since for those terms one would only need to average the drift term, one would only need to modify the definition of the average term $Y^\epsilon$, whereas the definition of the fluctuation term $Z^\epsilon$ would not be modified.

To conclude this article, let us mention that recently the averaging principle was proved for stochastic differential equations driven by a fractional Brownian motion with Hurst index $H>1/2$, see~\cite{HairerLi}:
\[
dX_t^\epsilon(t)=\sigma\bigl(X^\epsilon(t),m(t/\epsilon)\bigr)d\beta^H(t).
\]
The expression of the averaged equation is different from~\eqref{eq:averaged}: it is of the type
\[
d\overline{X}^H(t)=\overline{\sigma}(\overline{X}^{H}(t))d\beta^H(t)
\]
{\it i.e.} one simply needs to average the diffusion coefficient. In that case, the decomposition $X^\epsilon=Y^\epsilon+Z^\epsilon$ would give $Z^\epsilon\to 0$ when $\epsilon\to 0$, {\it i.e.} the fluctuation term does not contribute to the limit if $H>1/2$ -- in the same way as it does not contribute for drift terms. Our result thus illustrates the differences in the averaging principle between the standard and fractional Brownian motion cases. Note that, to the best of our knowledge, the validity and expression of the averaging principle if the Hurst index satisfies $H<1/2$ is not known. The approach introduced in this article may be suitable to investigate this challenging question in future works.

\section*{Acknowledgments}

The author would like to thank Greg Pavliotis and Andrew Stuart for the suggestions of some references.


\end{document}